\documentclass{amsart}
\usepackage{amsmath,amssymb,bm,amscd, amsthm}
\usepackage[dvips]{graphicx}
\theoremstyle{definition}
\newtheorem{defn}{Definition}[section]
\newtheorem{example}[defn]{Example}
\newtheorem{rem}[defn]{Remark}
\theoremstyle{plain}
\newtheorem{thm}[defn]{Theorem}
\newtheorem{prop}[defn]{Proposition}

\newtheorem{cor}[defn]{Corollary}
\newtheorem{question}[defn]{Question}

\newcommand{\KH}{\operatorname{KH}}

\numberwithin{equation}{section}
\newtheorem{proposition}{Proposition}[subsection]
\newtheorem{exam}[proposition]{Example}
\newtheorem{coro}[proposition]{Corollary}
\newtheorem{remark}[proposition]{Remark}
\title[On the maximal degree of the Khovanov homology]{On the maximal degree of the Khovanov homology}
\author{Keiji Tagami}
\date{\today}
\address{
Department of Mathematics, Faculty of Science and Technology, Tokyo University
of Science, Noda, Chiba, 278-8510, Japan
}
\email{tagami\_keiji@ma.noda.tus.ac.jp}
\begin{document}
\maketitle
\begin{abstract}
It is known that the maximal homological degree of the Khovanov homology of a knot gives a lower bound of the minimal positive crossing number of the knot. 
In this paper, we show that the maximal homological degree of the Khovanov homology of a cabling of a knot gives a lower bound of the minimal positive crossing number of the knot. 
\end{abstract}
%
\section{Introduction}\label{intro}
In \cite{khovanov1}, for each oriented link $L$, Khovanov defined a graded chain complex whose graded Euler characteristic is equal to the Jones polynomial of $L$. Its homology is a link invariant and called the Khovanov homology. 
Throughout this paper, we denote the homological degree $i$ term of the rational Khovanov homology of a link $L$ by $\KH^{i}(L)$.  
\par
Let $L$ be an oriented link. 
By $i_{\max}(L)$, we denote the maximal homological degree of the Khovanov homology of $L$, and 
by $c_{+}(L)$, we denote the minimal number of the positive crossings of diagrams of $L$. 
Note that $i_{\max}(L)$ is not negative. 
In fact, any link has nonzero Khovanov homology in degree zero because the Lee homology \cite{lee1} is not trivial in homological degree zero. 
Then, it is known that $i_{\max}(L)$ gives a lower bound of $c_{+}(L)$ (Proposition~$\ref{i_max}$). 
From this fact, it seems that the Khovanov homology estimates the positivity of links. 
\par
Sto{\v s}i{\'c} \cite{stosic2} showed that $i_{\max}(T_{2k, 2kt})$ is $2k^{2}t$, where $T_{p, q}$ is the positive $(p, q)$-torus link. 
By using the same method as Sto{\v s}i{\'c}, the author \cite{tagami1} proved that $i_{\max}(T_{2k+1, (2k+1)t})$ is $2k(k+1)t$. 
The author also computed the maximal degree for a cabling of any knot. 
\par
In this paper, we give some properties of the maximal degree of the Khovanov homology. 
In particular, we show that the maximal homological degree of the Khovanov homology of a cabling of a knot gives a lower bound of the minimal positive crossing number of the knot as follows: 
\begin{thm}\label{main}
Let $K$ be an oriented knot. 
Denote the $(p, q)$-cabling of $K$ by $K(p,q)$. 
For any positive integers $p$ and $t$, we assume that each component of $K(p,pt)$ has the orientation induced by $K$, that is, each component of $K(p,pt)$ is homologous to $K$ in the tubular neighborhood of $K$. 
Then, if $t\leq 2c_{+}(K)$, we have the following inequality: 
\begin{align*}
\frac{i_{\max}(K(p, pt))}{p^{2}}\leq c_{+}(K). 
\end{align*}
In particular, we obtain 
\begin{align*}
 \frac{i_{\max}(K(p, 0))}{p^{2}}\leq c_{+}(K).  
\end{align*}
\end{thm}
In many cases, $p^{2}c_{+}(K)$ is not greater than the minimal positive crossing number of $(p, 0)$-cabling of $K$. 
Hence the inequality of Theorem~\ref{main} is possibly stronger than that of Proposition~\ref{i_max}.
\par
Note that there are some works on the crossing numbers of cable links (for example, see \cite{Freedman-He}, \cite[Problem 1.68]{Kirby-problem} and \cite{Stoimenow4}). 
\par 
This paper is organized as follows: 
In Section~$\ref{property}$, we give the proof of Theorem~$\ref{main}$.
In Appendix, we introduce some properties of the maximal degree of the Khovanov homology.  
\par
We refer some informations (knot names, values of knot invariants and so on) in \cite{knot_info}. 
Throughout this paper, we use the same definition and notation of the Khovanov homology as in \cite[p.2848-p.2850]{tagami1}. 
In particular, for a link diagram $D$ of a link $L$, we denote the unnormalized Khovanov homology by $H^{i}(D)$, that is, 
\begin{center}
$\KH^{i}(L)=H^{i+c_{-}(D)}(D)$, 
\end{center}
where $c_{-}(D)$ is the number of the negative crossings of $D$. 
\section{The positivity of knots and the maximal degree of the Khovanov homology}\label{property}
In this section, we give some estimates on the minimal positive crossing numbers of knots. 
In particular, we prove Theorem~\ref{main}. 
\par
For any oriented link $L$, define 
\begin{align*}
i_{\max}(L)&:=\max\{i\in \mathbf{Z}\mid \KH^{i}(L)\neq 0\}, \\ 
i_{\min}(L)&:=\min\{i\in \mathbf{Z}\mid \KH^{i}(L)\neq 0\}, \\ 
c_{\pm }(L)&:=\min\{c_{\pm }(D)\mid D \text{\ is a diagram of }L\}, 
\end{align*}
where $c_{+}(D)$ ($c_{-}(D)$) is the number of the positive (negative) crossings of $D$. 
A link $L$ is {\it positive (negative)} if $c_{-}(L)=0$ ($c_{+}(L)=0$). 
The following is an immediate consequence of the definition of the Khovanov homology. 
\begin{prop}[{cf. \cite[Proposition~2.2]{tagami2}}]\label{i_max}
For any oriented link $L$, we have  
$i_{\max}(L)\leq c_{+}(L)$ and $-i_{\min}(L)\leq c_{-}(L)$. 
\end{prop}
%
%
%
%

Unfortunately, Proposition~\ref{i_max} is not sufficient to determine whether a given link is negative (or positive) (for example, see Example~\ref{example1}). 
In \cite{tagami1}, the author computed $i_{\max}$ for a cable link. 
By the computation, we obtain new estimates of the minimal positive crossing numbers of knots (Theorem~\ref{main}). 
\begin{proof}[Proof of Theorem~\ref{main}]
Let $D$ be a diagram of $K$ with $c_{+}(K)$ positive crossings. 
Let $c_{-}(D)$ be the number of the negative crossings of $D$. 
\par
Suppose that $p=2k$. 
By \cite[Lemma~$4.3$ $(1)$]{tagami1} (with $j=2k$ and $m=0$), we have 
\begin{center}
$H^{i}(D^{0}(2k, 2k(n+c_{+}(K)-c_{-}(D))+2k))=0$
\end{center}
for $i>(2k)^{2}(c_{+}(K)+c_{-}(D))$ and $0\leq n< c_{+}(K)+c_{+}(D)$.
\footnote{In \cite[Lemma~$4.3$]{tagami1}, we put $l:=c_{+}(D)+c_{+}(D)$ and $f:=c_{+}(D)-c_{-}(D)$. In our setting, $c_{+}(D)=c_{+}(K)$. }
Here, the diagram $D^{0}(2k, 2k(n+c_{+}(K)-c_{-}(D))+2k)$ is introduced in \cite[Definitions~3.1, 3.9 and 4.1, and Figures~$4$ and $8$]{tagami1}. 
Moreover $D^{0}(2k, 2k(n+c_{+}(K)-c_{-}(D))+2k)$ is equal to $D(2k, 2k(n+c_{+}(K)-c_{-}(D)+1))$, where $D(p, q)$ is a diagram of $K(p, q)$ introduced in \cite[Definition~4.1 and Figures~$7$ and 9]{tagami1}. 
Since the diagram $D(2k, 2k(n+c_{+}(K)-c_{-}(D)+1))$ has $(2k)^{2}c_{-}(D)$ negative crossings, by putting $t=n+c_{+}(K)-c_{-}(D)+1$, we obtain 
\[
\KH^{i}(K(2k, 2kt))=0 
\]
for $i>(2k)^{2}c_{+}(K)$ and $c_{+}(K)-c_{-}(D)+1\leq t\leq 2c_{+}(K)$. 
In particular, we have 
$ i_{\max}(K(2k, 2kt)) \leq (2k)^{2}c_{+}(K)$. 
By using the negative first Reidemeister move repeatedly, we can take as large a $c_{-}(D)$ as we want. 
Hence, for $-\infty< t\leq 2c_{+}(K)$,  we have 
\begin{align*}
 i_{\max}(K(2k, 2kt)) \leq (2k)^{2}c_{+}(K). 
\end{align*}
\par
Suppose that $p=2k+1$. 
By \cite[Lemma~$4.3$ $(2)$]{tagami1} and the same discussion as above, we obtain 
\begin{align*}
 i_{\max}(K(2k+1, (2k+1)t)) \leq (2k+1)^{2}c_{+}(K)
\end{align*}
for $-\infty< t\leq 2c_{+}(K)$. 
\end{proof}
\begin{cor}
If $K$ is a negative knot, then $i_{\max}(K(p, 0))$ is zero for any positive integer $p$. 
\end{cor}
%
%
%
\begin{example}\label{example1}
We see that $i_{\max}(8_{21})=i_{\max}(\overline{9_{45}})=i_{\max}(\overline{9_{46}})=0$ (see \cite{knot_info}). 
Hence, the inequality in Proposition~$\ref{i_max}$ cannot determine whether these knots are negative. 
However, by ``The Mathematica Package KnotTheory" \cite{Bar-Natan-2}, we have 
\[
i_{\max}(8_{21}(2,0))=i_{\max}(\overline{9_{45}}(2,0))=2, \ i_{\max}(\overline{9_{46}}(2,0))=4. 
\] 
By Theorem~$\ref{main}$, these knots are not negative. 
\end{example}
\begin{rem}
Let $K$ be an oriented knot. 
Then, $\KH^{i}(K(p, pt))=0$ if $i$, $n$ and $p$ satisfy one of the following conditions (see also Figure~\ref{fig1}): 
\begin{enumerate}
\item $i>p^2c_{+}(K)$ and $t\leq 2c_{+}(K)$, \label{enu1}
\item $p=2k$ for some $k>0$, $i>2k^2(t-2c_{+}(K))+p^2c_{+}(K)$ and $t>2c_{+}(K)$, \label{enu2}
\item $p=2k+1$ for some $k>0$, $i>2k(k+1)(t-2c_{+}(K))+p^2c_{+}(K)$ and $t> 2c_{+}(K)$. \label{enu3}
\end{enumerate}
Condition~(\ref{enu1}) follows from Theorem~\ref{main}, and (\ref{enu2}) and (\ref{enu3}) follow from \cite[Lemma~$4.3$]{tagami1}. 
\end{rem}
\begin{figure}[!htb]
\centering
\includegraphics[scale=0.35]{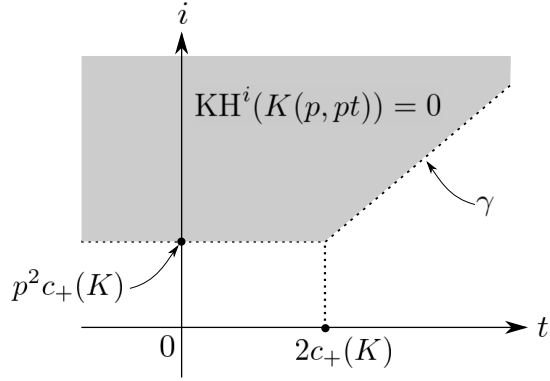}
\caption{$\KH^{i}(K(p, pt))=0$ in the gray area. The slope of the line $\gamma$ is $2k^2$ if $p=2k$ and $2k(k+1)$ if $p=2k+1$. }
\label{fig1}
\end{figure}
\begin{question}\label{ques:main}
For any non-negative knot $K$, are there some $p>0$ and $t\leq 2c_{+}(K)$ such that $i_{\max}(K(p, pt))>0$?
\end{question}
%
%
\begin{question}
For any knot $K$, does the following hold?
\begin{align*}
i_{\max}(K)\leq \frac{i_{\max}(K(p, 2pi_{\max}(K)))}{p^{2}} \leq c_{+}(K). 
\end{align*}
\end{question}
\noindent
Note that the last inequality holds by Theorem~\ref{main}. 
\appendix
\section{Appendix: Other properties of the Khovanov homology}\label{property2}
%
%
In this section, we introduce other properties of the maximal degree of the Khovanov homology. \par
%
%
%

\subsection{$i_{\max}$ versus $c_{+}$}
By the following result, the accuracy of the inequality in Proposition~$\ref{i_max}$ does not depend on the choice of orientations of links. 
\begin{proposition}
For any oriented link $L$, $c_{+}(L)-i_{\max}(L)$ does not depend on the orientation of $L$. 
\end{proposition}
\begin{proof}
Let $K$ be a component of $L$. 
Let $L'$ be the link obtained from $L$ by reversing the orientation of $K$. 
Khovanov \cite[Proposition~$28$]{khovanov1} proved that 
\begin{align*}
\KH^{i}(L')=\KH^{i+2lk(L\setminus K, K)}(L). 
\end{align*}
Hence we have $i_{\max}(L)=i_{\max}(L')+2lk(L\setminus K, K)$. 
On the other hand, we have 
$c_{+}(L)=c_{+}(L')+2lk(L\setminus K, K)$. 
These imply this proposition. 
\end{proof}
\begin{exam}\label{sample}
Let $T_{p, q}$ be the positive $(p,q)$-torus link. 
For positive integers $k$ and $t$, we have $c_{+}(T_{2k,2kt})=2kt(2k-1)$. 
Indeed, the standard diagram of $T_{2k, 2kt}$ has $2kt(2k-1)$ positive crossings. 
Moreover, each 2-component sublink of $T_{2k, 2kt}$ is $T_{2, 2t}$ and has at least $2t$ positive crossings. 
Since the link $T_{2k, 2kt}$ has $k(2k-1)$ 2-component sublinks, we obtain $c_{+}(T_{2k,2kt})\geq 2kt(2k-1)$. 
On the other hand, Sto{\v s}i{\'c} \cite[Theorem~$3$]{stosic2} showed that $i_{\max}(T_{2k, 2kt})=2k^{2}t$. 
Hence, we have $c_{+}(T_{2k, 2kt})-i_{\max}(T_{2k, 2kt})=2kt(k-1)$. 
\par
Let $T'_{2k, 2kt}$ be the link obtained from $T_{2k, 2kt}$ by reversing the orientations of exactly $k$ components. 
Similarly, we have $c_{+}(T'_{2k, 2kt})=2kt(k-1)$ and $i_{\max}(T'_{2k, 2kt})=0$. 
Hence we have $c_{+}(T'_{2k, 2kt})-i_{\max}(T'_{2k, 2kt})=2kt(k-1)$. 
\end{exam}
\begin{coro}
For any positive integer $N$, there exists some oriented link $L$ such that $c_{+}(L)-i_{\max}(L)>N$. 
\end{coro}
\begin{proof}
By Example~$\ref{sample}$, we have $c_{+}(T_{4,4N})-i_{\max}(T_{4, 4N})>N$. 
\end{proof}
\subsection{Additivity of $i_{\max}$}
We prove the additivity of the maximal degree of the Khovanov homology. 
\begin{proposition}
For any oriented knots $K_{1}$ and $K_{2}$, we have  
\begin{align*}
i_{\max}(K_{1}\sharp K_{2})=i_{\max}(K_{1}\sqcup K_{2})=i_{\max}(K_{1})+i_{\max}(K_{2}), 
\end{align*}
where $K_{1}\sqcup K_{2}$ is the disjoint union of the knots and $K_{1}\sharp K_{2}$ is the connected sum of the knots. 
\end{proposition}
\begin{proof}
By \cite[Proposition~$33$]{khovanov1}, we obtain 
\begin{align*}
\KH^{i}(K_{1}\sqcup K_{2})\cong \bigoplus_{p+q=i}\KH^{p}(K_{1})\otimes \KH^{q}(K_{2}).
\end{align*} 
Note that our coefficient ring is the rational number field $\mathbf{Q}$. 
Hence we have $i_{\max}(K_{1}\sqcup K_{2})=i_{\max}(K_{1})+i_{\max}(K_{2})$. 
\par
In \cite[Proposition~$34$]{khovanov1}, the following exact sequence was introduced: 
\begin{multline*}
\cdots \rightarrow \KH^{i}(K_{1}\sharp K_{2})\rightarrow \KH^{i}(K_{1}\sqcup K_{2})\rightarrow \KH^{i}(K_{1}\sharp K_{2}) \\
\rightarrow \KH^{i+1}(K_{1}\sharp K_{2})\rightarrow \KH^{i+1}(K_{1}\sqcup K_{2})\rightarrow KH^{i+1}(K_{1}\sharp K_{2}) \\
\rightarrow \KH^{i+2}(K_{1}\sharp K_{2})\rightarrow \KH^{i+2}(K_{1}\sqcup K_{2})\rightarrow \cdots.
\end{multline*}
Here, we forget the quantum grading of the Khovanov homology. 
By the first line, we have $\KH^{i}(K_{1}\sharp K_{2})\neq 0$ if $\KH^{i}(K_{1}\sqcup K_{2})\neq 0$. 
Hence we obtain $i_{\max}(K_{1}\sharp K_{2})\geq i_{\max}(K_{1}\sqcup K_{2})$. 
\par
Put $i_{0}=i_{\max}(K_{1}\sqcup K_{2})$. Then $\KH^{i_{0}+1}(K_{1}\sqcup K_{2})=\KH^{i_{0}+2}(K_{1}\sqcup K_{2})=0$. 
By the second and third lines, we have 
\[
\KH^{i_{0}+1}(K_{1}\sharp K_{2})=\KH^{i_{0}+2}(K_{1}\sharp K_{2}). 
\]
By repeating this process, for any $l\in\mathbf{Z}_{> 0}$, we obtain 
\[
\KH^{i_{0}+1}(K_{1}\sharp K_{2})=\KH^{i_{0}+l}(K_{1}\sharp K_{2}). 
\]
Since $\KH^{i_{0}+l}(K_{1}\sharp K_{2})=0$ for sufficiently large $l$, we have 
\[
\KH^{i_{0}+l}(K_{1}\sharp K_{2})=0, 
\]
for any $l\in\mathbf{Z}_{>0}$. 
This implies that $i_{0}\geq i_{\max}(K_{1}\sharp K_{2})$. 
Hence we obtain $i_{\max}(K_{1}\sharp K_{2})= i_{\max}(K_{1}\sqcup K_{2})=i_{\max}(K_{1})+i_{\max}(K_{2})$. 
\end{proof}
\subsection{Almost positive knots}
A link diagram is {\it almost positive (negative)} if it has exactly one negative (positive) crossing. 
A link is {\it almost positive (negative)} if it is not a positive (negative) knot and has an almost positive (negative) diagram. 
Then we have the following. 
\begin{coro}\label{almost1}
For any almost negative link $L$, we have $i_{\max}(L)=0$. 
\end{coro}
\begin{proof}
By Proposition~$\ref{i_max}$, we have $i_{\max}(L)=0$ or $1$. 
For contradiction, assume that $i_{\max}(L)=1$. 
Then, any almost negative diagram $D$ of $L$ satisfies $i_{\max}(L)=c_{+}(D)=1$. 
In \cite[Proposition~$36$]{khovanov1}, Khovanov proved that $i_{\max}(L)=c_{+}(D)$ if and only if the diagram $D$ is $+$adequate (for the definition of $+$adequate diagrams, for example, see \cite[Definition~$3$]{khovanov1}). 
However, any reduced almost negative link diagram is not $+$adequate. 
Hence, $i_{\max}(L)=0$. 
\end{proof}
\begin{remark}
The Khovanov homology and the Rasmussen invariant for an almost positive link are studied in \cite{abe-tagami1} and \cite{tagami7}. 
\end{remark}
\noindent{\bf Acknowledgements: }
The author would like to thank the referees for their helpful comments. 
\bibliographystyle{amsplain}
\bibliography{mrabbrev,tagami}

\providecommand{\bysame}{\leavevmode\hbox to3em{\hrulefill}\thinspace}
\providecommand{\MR}{\relax\ifhmode\unskip\space\fi MR }
\providecommand{\MRhref}[2]{%
  \href{http://www.ams.org/mathscinet-getitem?mr=#1}{#2}
}
\providecommand{\href}[2]{#2}
\begin{thebibliography}{10}

\bibitem{abe-tagami1}
T.~Abe and K.~Tagami, \emph{Characterization of positive links and the $ s
  $-invariant for links}, to appear in Canad. J. Math.

\bibitem{Bar-Natan-2}
D.~Bar-Natan, \emph{The {K}not {A}tlas},
  http://www.math.toronto.edu/drorbn/KAtlas/.

\bibitem{Freedman-He}
M.~H. Freedman and Z.~He, \emph{Divergence-free fields: energy and asymptotic
  crossing number}, Ann. of Math. (2) \textbf{134} (1991), no.~1, 189--229.
  \MR{1114611 (93a:58040)}

\bibitem{khovanov1}
M.~Khovanov, \emph{A categorification of the {J}ones polynomial}, Duke Math. J.
  \textbf{101} (2000), no.~3, 359--426. \MR{1740682 (2002j:57025)}

\bibitem{Kirby-problem}
R.~Kirby, \emph{Problems in low-dimensional topology}, Geometric topology
  ({A}thens, {GA}, 1993), AMS/IP Stud. Adv. Math., vol.~2, Amer. Math. Soc.,
  Providence, RI, 1997, pp.~35--473. \MR{1470751}

\bibitem{lee1}
E.~S. Lee, \emph{An endomorphism of the {K}hovanov invariant}, Adv. Math.
  \textbf{197} (2005), no.~2, 554--586. \MR{2173845 (2006g:57024)}

\bibitem{knot_info}
C.~Livingston and J.~Cha, \emph{Knot {I}nfo},
  http://www.indiana.edu/\%7eknotinfo/.

\bibitem{Stoimenow4}
A.~Stoimenow, \emph{On the satellite crossing number conjecture}, J. Topol.
  Anal. \textbf{3} (2011), no.~2, 109--143. \MR{2819190 (2012g:57017)}

\bibitem{stosic2}
M.~Sto{\v{s}}i{\'c}, \emph{Khovanov homology of torus links}, Topology Appl.
  \textbf{156} (2009), no.~3, 533--541. \MR{2492301 (2011b:57008)}

\bibitem{tagami1}
K.~Tagami, \emph{The maximal degree of the {K}hovanov homology of a cable
  link}, Algebr. Geom. Topol. \textbf{13} (2013), no.~5, 2845--2896.
  \MR{3116306}

\bibitem{tagami7}
\bysame, \emph{The {R}asmussen {I}nvariant, {F}our-genus and {T}hree-genus of
  an {A}lmost {P}ositive {K}not {A}re {E}qual}, Canad. Math. Bull. \textbf{57}
  (2014), no.~2, 431--438. \MR{3194190}

\bibitem{tagami2}
\bysame, \emph{The behavior of the maximal degree of the {K}hovanov homology
  under twisting}, Topology. Proc. \textbf{46} (2015), 45--54.

\end{thebibliography}
\end{document}